\documentclass[11pt,bezier]{article}
\usepackage{amsmath,amssymb,amsfonts,enumerate,setspace,graphicx,tkz-berge}

\newtheorem{prethm}{{\bf  Theorem}}

\newenvironment{thm}{\begin{prethm}{\hspace{-0.5
               em}{\bf .}}}{\end{prethm}}

\newtheorem{prepro}{{\bf  Theorem}}

\newenvironment{pro}{\begin{prepro}{\hspace{-0.5
               em}{\bf .}}}{\end{prepro}}
\newtheorem{precor}{{\bf  Corollary}}

\newenvironment{cor}{\begin{precor}{\hspace{-0.5
               em}{\bf .}}}{\end{precor}}
\newtheorem{preconj}{{\bf  Conjecture}}

\newtheorem{preremark}{{\bf  Remark}}

\newenvironment{remark}{\begin{preremark}{\hspace{-0.5
               em}{\bf .}}}{\end{preremark}}
\newtheorem{prelem}{{\bf  Lemma}}

\newenvironment{lem}{\begin{prelem}{\hspace{-0.5
               em}{\bf .}}}{\end{prelem}}
\newtheorem{preproof}{{\bf  Proof.}}

\newenvironment{proof}[1]{\begin{preproof}{\rm
               #1}\hfill{$\Box$}}{\end{preproof}}

\title{\large \ {\bf On the intersection graph of ideals of $\mathbb{Z}_m$}
\thanks
{
{\indent ~~{2010{ \it Mathematics Subject Classification}: 05C15, 05C25, 05C45, 05C69.}}\newline
\newline
{\indent ~~{{\it E-mail addresses}: s\_khojasteh@liau.ac.ir (S. Khojasteh).}}}}

\author{\bf\small\sc  S. Khojasteh \\
\\ {\footnotesize {\em Department of Mathematics,  Lahijan Branch, Islamic Azad University, Lahijan, Iran}}}

\date{}

\begin{document}
\maketitle
\begin{abstract}
 Let $m>1$ be an integer, and let $I(\mathbb{Z}_m)^*$ be the set of all non-zero proper ideals of $\mathbb{Z}_m$. The intersection graph of ideals of $\mathbb{Z}_m$, denoted by $G(\mathbb{Z}_m)$, is a graph with vertices $I(\mathbb{Z}_m)^*$ and two distinct vertices $I,J\in I(\mathbb{Z}_m)^*$ are adjacent if and only if $I\cap J\neq 0$. Let $n>1$ be an integer and $\mathbb{Z}_n$ be a $\mathbb{Z}_m$-module. In this paper, we introduce
and study a kind of graph structure of $\mathbb{Z}_m$, denoted by $G_n(\mathbb{Z}_m)$. It is the undirected graph with the vertex set $I(\mathbb{Z}_m)^*$, and two distinct vertices $I$ and $J$ are adjacent if and only if $I\mathbb{Z}_n\cap J\mathbb{Z}_n\neq 0$. Clearly, $G_m(\mathbb{Z}_m)=G(\mathbb{Z}_m)$. We obtain some graph theoretical properties of $G_n(\mathbb{Z}_m)$ and we compute some of its numerical invariants, namely girth, independence number, domination number, maximum degree and chromatic index. We also determine all integer numbers $n$ and $m$ for which $G_n(\mathbb{Z}_m)$ is Eulerian.
\end{abstract}

\section{ Introduction}
 Let $R$ be a commutative ring, and $I(R)^*$ be the set of all non-zero proper ideals of $R$.
There are many papers on assigning a graph to a ring $R$, for instance see \cite{unicyclic},
\cite{conjecture}, \cite{and2} and \cite{atani}.
Also the intersection graphs of some algebraic structures such as groups, rings and modules have been studied by several authors, see \cite{akbtaval, Chakrabarty, Cs}.
 In \cite{Chakrabarty}, the intersection graph of ideals of $R$, denoted by $G(R)$, was introduced as the graph with vertices $I(R)^*$ and for distinct $I,J\in I(R)^*$,
 the vertices $I$ and $J$ are adjacent if and only if $I\cap J\neq 0$. Also in \cite{akbtaval}, the intersection
graph of submodules of an $R$-module $M$, denoted by $G(M)$, is defined to be the graph whose vertices are
the non-zero proper submodules of $M$ and two distinct vertices are
adjacent if and only if they have non-zero intersection.
Let $n,m>1$ be integers and $\mathbb{Z}_n$ be a $\mathbb{Z}_m$-module. In this paper, we associate a graph to $\mathbb{Z}_m$, in which the vertex set is being the set of all non-zero proper ideals of $\mathbb{Z}_m$, and two distinct vertices $I$ and $J$ are adjacent if and only if $I\mathbb{Z}_n\cap J\mathbb{Z}_n\neq 0$. We denote this graph by $G_n(\mathbb{Z}_m)$. Clearly, if $n=m$, then $G_n(\mathbb{Z}_m)$ is exactly the same as the intersection graph of ideals of $\mathbb{Z}_m$. This implies that $G_n(\mathbb{Z}_m)$ is a generalization of $G(\mathbb{Z}_m)$. As usual, $\mathbb{Z}_m$ denotes the integers modulo $m$.\\
\indent Now, we recall some definitions and notations on graphs. Let $G$ be a graph with the vertex set $V(G)$ and the edge set $E(G)$. Then we say the order of $G$ is $|V(G)|$ and the size of $G$ is $|E(G)|$. Suppose that $x,y\in V(G)$. If $x$ and $y$ are adjacent, then we write $x$ --- $y$. We denote by $deg(x)$ the degree of a vertex $x$ in $G$. Also, we denote the maximum degree of $G$ by $\Delta(G)$. We recall
that a \textit{path} between $x$ and $y$ is a sequence $x=v_0$ --- $v_1$ --- $\cdots$ --- $v_k=y$ of vertices of $G$ such that for every $i$ with $1 \leq i \leq k$, the vertices $v_{i-1}$ and $v_{i}$ are adjacent and $v_{i}\neq v_{j}$, where $i\neq j$. We say that $G$ is \textit{connected} if there is a path between any two distinct vertices of $G$. For vertices $x$ and $y$ of $G$, let $d(x,y)$ be the length of a shortest path from $x$ to $y$ ($d(x,x)=0$ and $d(x,y)=\infty$ if there is no path between $x$ and $y$). The \textit{diameter} of $G$, $diam(G)$, is the supremum of the set $\{d(x,y) : x \ \hbox{and} \ y \ \hbox{are vertices of} \ G\}$. The \textit{girth} of $G$, denoted by $gr(G)$, is the length of a shortest cycle in $G$ ($gr(G)=\infty$ if $G$ contains no cycles). We use $n$-cycle to denote the cycle with $n$ vertices, where $n \geq 3$. Also, we denote the complete graph on $n$ vertices by $K_n$. A \textit{null graph} is a graph containing no edges. We use $\overline{K_n}$ to denote the null graph of order $n$. The \textit{disjoint union} of two vertex-disjoint graphs $G_{1}$ and $G_{2}$, which is denoted by $G_{1}\cup G_{2}$, is a graph with $V(G_{1}\cup G_{2})=V(G_{1})\cup V(G_{2})$ and $E(G_{1}\cup G_{2})=E(G_{1})\cup E(G_{2})$. An {\it independent set} is a subset of the vertices of a graph such that no vertices are adjacent. The number of vertices in a maximum independent set of $G$ is called the {\it independence number} of $G$ and is denoted by $\alpha(G)$. A {\it dominating set} is a subset $S$ of $V(G)$ such that every vertex of $V(G)\setminus S$ is adjacent to at least one vertex in $S$. The number of vertices in a smallest dominating set denoted by $\gamma(G)$, is called the {\it domination number} of $G$. Recall that a $k$-edge coloring of $G$ is an assignment of $k$ colors $\{1, \ldots, k\}$ to the edges of $G$ such that no two adjacent edges have the same color, and the \textit{chromatic index} of $G$, $\chi'(G)$, is the smallest integer $k$ such that $G$ has a $k$-edge coloring.\\
\indent In \cite{Chakrabarty}, the authors were mainly interested in the study of
intersection graph of ideals of $\mathbb{Z}_m$. For instance, they determined the values of $m$ for which $G(\mathbb{Z}_m)$ is connected, complete, Eulerian or has a cycle. In this article, we generalize these results to $G_n(\mathbb{Z}_m)$ and also, we find some new results. In Section 2, we compute its girth, independence number, domination number and maximum degree. We also determine all integer numbers $n$ and $m$ for which $G_n(\mathbb{Z}_m)$ is a forest. In Section 3, we investigate the chromatic index of $G_n(\mathbb{Z}_m)$. In the last section, we determine all integer numbers $n$ and $m$ for which $G_n(\mathbb{Z}_m)$ is Eulerian.

\section{ Basic Properties of $G_n(\mathbb{Z}_m)$}

Let $n,m>1$ be integers and $\mathbb{Z}_n$ be a $\mathbb{Z}_m$-module. Clearly, $\mathbb{Z}_n$ is a $\mathbb{Z}_m$-module if and only if $n$ divides $m$. Throughout the paper, without loss of generality, we assume that $m=p_1^{\alpha_1}\cdots p_s^{\alpha_s}$ and $n=p_1^{\beta_1}\cdots p_s^{\beta_s}$, where $p_i$'s are distinct primes, $\alpha_i$'s are positive integers, $\beta_i$'s are non-negative integers, and $0\leq \beta_i\leq \alpha_i$ for $i=1,\ldots,s$. Let $S=\{1,\ldots,s\}$, $S'=\{i\in S \ :\,\beta _i\neq 0\}$. The cardinality of $S'$ is denoted by $s'$. Also, we denote the least common multiple of integers $a$ and $b$ by $[a,b]$. We write $a|b$ ($a\nmid b$) if $a$ divides $b$ ($a$ does not divide $b$). We begin with a simple example.

\par\noindent{\bf Example 1.}
Let $m=12$. Then we have the following graphs.
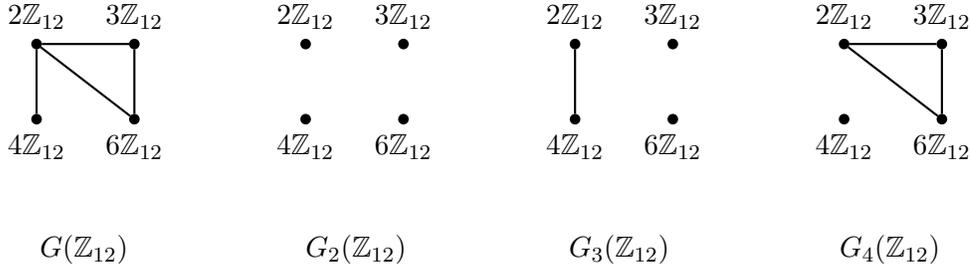
\begin{center}
    \begin{tikzpicture}
        \GraphInit[vstyle=Classic]
        \Vertex[x=1,y=0,style={black,minimum size=3pt},LabelOut=true,Lpos=270,L=$4\mathbb{Z}_{12}$]{4}
        \Vertex[x=1,y=1,style={black,minimum size=3pt},LabelOut=true,Lpos=90,L=$2\mathbb{Z}_{12}$]{2}
        \Vertex[x=2.3,y=0,style={black,minimum size=3pt},LabelOut=true,Lpos=270,L=$6\mathbb{Z}_{12}$]{6}
        \Vertex[x=2.3,y=1,style={black,minimum size=3pt},LabelOut=true,Lpos=90,L=$3\mathbb{Z}_{12}$]{3}
        \Edges(2,3)
        \Edges(6,3)
        \Edges(2,6)
        \Edges(2,4)

    \end{tikzpicture}
\hspace{1cm}
    \begin{tikzpicture}
        \GraphInit[vstyle=Classic]
        \Vertex[x=1,y=0,style={black,minimum size=3pt},LabelOut=true,Lpos=270,L=$4\mathbb{Z}_{12}$]{4}
        \Vertex[x=1,y=1,style={black,minimum size=3pt},LabelOut=true,Lpos=90,L=$2\mathbb{Z}_{12}$]{2}
        \Vertex[x=2.3,y=0,style={black,minimum size=3pt},LabelOut=true,Lpos=270,L=$6\mathbb{Z}_{12}$]{6}
        \Vertex[x=2.3,y=1,style={black,minimum size=3pt},LabelOut=true,Lpos=90,L=$3\mathbb{Z}_{12}$]{3}

    \end{tikzpicture}
\hspace{1cm}
    \begin{tikzpicture}
        \GraphInit[vstyle=Classic]
        \Vertex[x=1,y=0,style={black,minimum size=3pt},LabelOut=true,Lpos=270,L=$4\mathbb{Z}_{12}$]{4}
        \Vertex[x=1,y=1,style={black,minimum size=3pt},LabelOut=true,Lpos=90,L=$2\mathbb{Z}_{12}$]{2}
        \Vertex[x=2.3,y=0,style={black,minimum size=3pt},LabelOut=true,Lpos=270,L=$6\mathbb{Z}_{12}$]{6}
        \Vertex[x=2.3,y=1,style={black,minimum size=3pt},LabelOut=true,Lpos=90,L=$3\mathbb{Z}_{12}$]{3}
        \Edges(2,4)

    \end{tikzpicture}
\hspace{1cm}
    \begin{tikzpicture}
        \GraphInit[vstyle=Classic]
        \Vertex[x=1,y=0,style={black,minimum size=3pt},LabelOut=true,Lpos=270,L=$4\mathbb{Z}_{12}$]{4}
        \Vertex[x=1,y=1,style={black,minimum size=3pt},LabelOut=true,Lpos=90,L=$2\mathbb{Z}_{12}$]{2}
        \Vertex[x=2.3,y=0,style={black,minimum size=3pt},LabelOut=true,Lpos=270,L=$6\mathbb{Z}_{12}$]{6}
        \Vertex[x=2.3,y=1,style={black,minimum size=3pt},LabelOut=true,Lpos=90,L=$3\mathbb{Z}_{12}$]{3}
        \Edges(2,3)
        \Edges(6,3)
        \Edges(2,6)

    \end{tikzpicture}
 \end{center}
\hspace{1.4cm}
$G(\mathbb{Z}_{12})$ \hspace{2.1cm}
$G_2(\mathbb{Z}_{12})$ \hspace{1.9cm}
$G_3(\mathbb{Z}_{12})$ \hspace{2cm}
$G_4(\mathbb{Z}_{12})$

\begin{remark}
\label{zngzm}
{\rm  It is easy to see that $I(\mathbb{Z}_m)=\{d\mathbb{Z}_m : d$ divides $m \}$ and
$|I(\mathbb{Z}_m)^{*}|=\prod_{i=1}^{s}(\alpha_i+1)-2$. Let $\mathbb{Z}_n$ be a $\mathbb{Z}_m$-module. If $n|d$, then $d\mathbb{Z}_m$ is an isolated vertex of $G_n(\mathbb{Z}_m)$.
Obviously, $d_1\mathbb{Z}_m$ and $d_2\mathbb{Z}_m$ are adjacent in $G_n(\mathbb{Z}_m)$ if and only if $n\nmid [d_1,d_2]$. This implies that $G_n(\mathbb{Z}_m)$ is a subgraph of $G(\mathbb{Z}_m)$.}
\end{remark}

By \cite[Theorem 2.5]{akbtaval}, we have $gr(G(\mathbb{Z}_m))\in\{3,\infty\}$. We extend this result to $G_n(\mathbb{Z}_m)$.
\begin{thm}\label{girth}
{ Let $\mathbb{Z}_n$ be a $\mathbb{Z}_m$-module. Then $gr(G_n(\mathbb{Z}_m))\in\{3,\infty\}.$}
\end{thm}

\begin{proof}
{ With no loss of generality assume that $S'=\{1,\ldots,s'\}$. Clearly, if $s'\geq 3$, then $p_1\mathbb{Z}_m$ --- $p_2\mathbb{Z}_m$ --- $p_1p_2\mathbb{Z}_m$ is a $3$-cycle in $G_n(\mathbb{Z}_m)$. Therefore $gr(G_n(\mathbb{Z}_m))=3$. Now, consider two following cases:\\
\indent {Case 1.} $s'=2$. If $s\geq 3$, then $p_1\mathbb{Z}_m$ --- $p_3\mathbb{Z}_m$ --- $p_1p_3\mathbb{Z}_m$ is a $3$-cycle in $G_n(\mathbb{Z}_m)$. So we may assume that $s=2$. If $\alpha_i\geq 3$ for some $i$, $i=1,2$, then $p_i\mathbb{Z}_m$ --- $p_i^2\mathbb{Z}_m$ --- $p_i^3\mathbb{Z}_m$ is a $3$-cycle in $G_n(\mathbb{Z}_m)$. Also, if $\beta_i\geq 2$ for some $i$, $i=1,2$, then $p_1\mathbb{Z}_m$ --- $p_2\mathbb{Z}_m$ --- $p_1p_2\mathbb{Z}_m$ is a $3$-cycle in $G_n(\mathbb{Z}_m)$. Now, assume that $n=p_1p_2$ and $\alpha_1, \alpha_2=1,2$. It is easy to see that $gr(G_n(\mathbb{Z}_m))=\infty$. Note that $G_{p_1p_2}(\mathbb{Z}_{p_1p_2})\cong \overline{K_2}$, $G_{p_1p_2}(\mathbb{Z}_{p_1^2p_2})\cong K_2\cup \overline{K_2}$, and $G_{p_1p_2}(\mathbb{Z}_{p_1^2p_2^2})\cong K_2\cup K_2\cup \overline{K_3}$. \\
\indent {Case 2.} $s'=1$. If $s\geq 3$, then $p_2\mathbb{Z}_m$ --- $p_3\mathbb{Z}_m$ --- $p_2p_3\mathbb{Z}_m$ is a $3$-cycle in $G_n(\mathbb{Z}_m)$. So assume that $s=2$. If $\beta_1\geq 2$, then $p_1\mathbb{Z}_m$ --- $p_2\mathbb{Z}_m$ --- $p_1p_2\mathbb{Z}_m$ is a $3$-cycle in $G_n(\mathbb{Z}_m)$. Also, if $\alpha_2\geq 3$, then $p_2\mathbb{Z}_m$ --- $p_2^2\mathbb{Z}_m$ --- $p_2^3\mathbb{Z}_m$ is a $3$-cycle in $G_n(\mathbb{Z}_m)$. Now, suppose that $n=p_1$ and $m=p_1^{\alpha_1}p_2$ or $m=p_1^{\alpha_1}p_2^2$. Then
$gr(G_n(\mathbb{Z}_m))=\infty$, since $G_{p_1}(\mathbb{Z}_{p_1^{\alpha_1}p_2})\cong \overline{K_{2\alpha_1}}$ and $G_{p_1}(\mathbb{Z}_{p_1^{\alpha_1}p_2^2})\cong \overline{K_{3\alpha_1-1}}\cup K_2$. Finally, assume that $n=p_1^{\beta_1}$ and $m=p_1^{\alpha_1}$. If $\beta_1\geq 4$, then $p_1\mathbb{Z}_m$ --- $p_1^2\mathbb{Z}_m$ --- $p_1^3\mathbb{Z}_m$ is a $3$-cycle in $G_n(\mathbb{Z}_m)$. It is easy to see that if $\beta_1 \leq 3$, then $gr(G_{p_1^{\beta_1}}(\mathbb{Z}_{p_1^{\alpha_1}}))=\infty$.
}
\end{proof}

As an immediate consequence of Theorem \ref{girth}, we have the following corollary.
\begin{cor}
\label{forest}
{Let $\mathbb{Z}_n$ be a $\mathbb{Z}_m$-module. Then $G_n(\mathbb{Z}_m)$ is a forest if and only if one of the following holds:
\par $(i)$ $n=p_1p_2$, $m=p_1^{\alpha_1}p_2^{\alpha_2}$ and $\alpha_1, \alpha_2 \leq 2$.
\par $(ii)$ $n=p_1$, $m=p_1^{\alpha_1}p_2^{\alpha_2}$ and $\alpha_2 \leq 2$.
\par $(iii)$ $n=p_1^{\beta_1}$, $m=p_1^{\alpha_1}$ and $1\leq \beta_1 \leq 3$.
}
\end{cor}

By \cite[Theorem 3.4]{akbtaval}, we find that $G(\mathbb{Z}_m)$ is a tree if and only if $G(\mathbb{Z}_m)$ is a star. Now, we have a similar result.
\begin{cor}
\label{tree}
{ Let $\mathbb{Z}_n$ be a $\mathbb{Z}_m$-module. Then $G_n(\mathbb{Z}_m)$ is a tree if and only if $G_n(\mathbb{Z}_m)$ is a star. In particular, $G_n(\mathbb{Z}_m)$ is a tree if and only if one of the following holds:
\par $(i)$ $n=m=p_1^2$.
\par $(ii)$ $n=m=p_1^3$.
\par $(iii)$ $n=p_1$ and $m=p_1^2$.
}
\end{cor}

By Corollary \ref{forest}, we can characterize the values of $n$ and $m$ for which $G_n(\mathbb{Z}_m)$ is a null graph.
\begin{cor}
\label{nullgraph}
{ Let $\mathbb{Z}_n$ be a $\mathbb{Z}_m$-module. Then $G_n(\mathbb{Z}_m)$ is a null graph if and only if one of the following holds:
\par $(i)$ $n=p_1$ and $m=p_1^{\alpha_1}p_2$.
\par $(ii)$ $n=p_1$ and $m=p_1^{\alpha_1}$.
\par $(iii)$ $n=p_1^{2}$ and $m=p_1^{\alpha_1}$.
\par $(iv)$ $n=m=p_1p_2$.}
\end{cor}

Throughout the rest of this paper, we use $\mathfrak{A}$ to denote the set of all isolated vertices of $G_n(\mathbb{Z}_m)$.
\begin{lem}
\label{n,d,isole}
{ Let $\mathbb{Z}_n$ be a $\mathbb{Z}_m$-module. If $G_n(\mathbb{Z}_m)$ is not a null graph, then $d\mathbb{Z}_m$ is an isolated vertex of $G_n(\mathbb{Z}_m)$ if and only if $n|d$, except for the case $n=p_1p_2$, $m=p_1^{\alpha_1}p_2$ and $\alpha_1\geq 2$, in which case $G_n(\mathbb{Z}_m)=K_{\alpha_1}\cup \overline{K_{\alpha_1}}$.}
\end{lem}
\begin{proof}
{ Clearly, if $n|d$, then $d\mathbb{Z}_m$ is an isolated vertex of $G_n(\mathbb{Z}_m)$. For the other side suppose that $G_n(\mathbb{Z}_m)$ is not a null graph, $d=p_1^{r_1}\cdots p_s^{r_s}$ is a divisor of $m$ and $n\nmid d$. Since $n\nmid d$, we may assume that $r_1< \beta_{1}$. If $s\geq 3$, then $d\mathbb{Z}_m$ is adjacent to one of $p_2^{\alpha_2}\cdots p_s^{\alpha_s}\mathbb{Z}_m$ or $p_3^{\alpha_3}\cdots p_s^{\alpha_s}\mathbb{Z}_m$. Now, suppose that $s=2$. If $r_1\neq 0$, then $d\mathbb{Z}_m$ and $p_2^{\alpha_2}\mathbb{Z}_m$ are adjacent. Hence $r_1=0$ and $d=p_2^{r_2}$. If $r_2\geq 2$, then $d\mathbb{Z}_m$ and $p_2\mathbb{Z}_m$ are adjacent. Therefore $d=p_2$. If $\beta_1\geq 2$, then $d\mathbb{Z}_m$ and $p_1\mathbb{Z}_m$ are adjacent. Thus $\beta_1=1$. If $\alpha _2\geq 2$, then $d\mathbb{Z}_m$ and $p_2^{2}\mathbb{Z}_m$ are adjacent. So $\alpha_2 =1$. Then $m=p_1^{\alpha_1}p_2$ and $n=p_1$ or $n=p_1p_2$. If $n=p_1$, then by Corollary \ref{nullgraph}, $G_n(\mathbb{Z}_m)$ is a null graph, a contradiction. Therefore $n=p_1p_2$, $m=p_1^{\alpha_1}p_2$ and $\alpha_1\geq 2$, in which case $d\mathbb{Z}_m$ is an isolated vertex of $G_n(\mathbb{Z}_m)$. Moreover, $\mathfrak{A}= \{p_1^{r_1}p_2\mathbb{Z}_m : 0\leq r_1\leq \alpha_1-1\}$ and $V(G_n(\mathbb{Z}_m))\setminus \mathfrak{A}=\{p_1^{r_1}\mathbb{Z}_m : 1\leq r_1\leq \alpha_1\}$. Clearly, $p_1^{r_1}\mathbb{Z}_m$ and $p_1^{r_2}\mathbb{Z}_m$ are adjacent, where $1\leq r_1<r_2\leq \alpha_1$. This implies that $G_{p_1p_2}(\mathbb{Z}_{p_1^{\alpha_1}p_2})=K_{\alpha_1}\cup \overline{K_{\alpha_1}}$, where $\alpha_1\geq 2$. Next, suppose that $s=1$. Since $G_n(\mathbb{Z}_m)$ is not a null graph, by Corollary \ref{nullgraph}, we conclude that $\beta_1\geq3$. Therefore $d\mathbb{Z}_m$ is adjacent to $p_1\mathbb{Z}_m$ or $p_1^{2}\mathbb{Z}_m$. This completes the proof.}
\end{proof}

\begin{lem}
\label{cardinalisole}
{ Let $\mathbb{Z}_n$ be a $\mathbb{Z}_m$-module. If $G_n(\mathbb{Z}_m)$ is not a null graph, then $$|\mathfrak{A}|=\left\{
                  \begin{array}{ll}
                    \alpha_1, & \hbox{if $n= p_1p_2$, $m=p_1^{\alpha_1}p_2$ and $\alpha_1\geq 2$;} \\
                    \prod_{i=1}^{s}(\alpha_i-\beta_i+1)-1, & \hbox{otherwise.}
                  \end{array}
                \right.
                $$}
\end{lem}
\begin{proof}
{ By Lemma \ref{n,d,isole}, we know that if $n= p_1p_2$, $m=p_1^{\alpha_1}p_2$ and  $\alpha_1\geq 2$, then $|\mathfrak{A}|=\alpha_1$. Otherwise, from Lemma \ref{n,d,isole}, we find that $d\mathbb{Z}_m$ is an isolated vertex of $G_n(\mathbb{Z}_m)$ if and only if $n$ divides $d$. So $\mathfrak{A}=\{d\mathbb{Z}_m:n$ divides $d$, $d$ divides $m$, $d\neq m\}=\{p_1^{r_1}\cdots p_s^{r_s}\mathbb{Z}_m : \beta _i\leq r_i\leq \alpha _i\}\setminus\{0\}$. This implies that $|\mathfrak{A}|=\prod_{i=1}^{s}(\alpha_i-\beta_i+1)-1$ and the proof is complete.}
\end{proof}

\begin{cor}\label{noisolatedvertex}
{ Let $\mathbb{Z}_n$ be a $\mathbb{Z}_m$-module. Then $G_n(\mathbb{Z}_m)$ contains no isolated vertex if and only if $n=m\neq p_1, p_1^2, p_1p_2$.}
\end{cor}

\begin{proof}
{ One side is obvious. For the other side assume that $G_n(\mathbb{Z}_m)$ contains no isolated vertex. Hence $n=m$. By Corollary \ref{nullgraph} and Lemma \ref{cardinalisole}, it is clear that $m\neq p_1,p_1^2, p_1p_2$.
}
\end{proof}

\begin{thm}
{ Let $\mathbb{Z}_n$ be a $\mathbb{Z}_m$-module. If $G_n(\mathbb{Z}_m)$ is not a null graph, then
\begin{center}
$\alpha(G_n(\mathbb{Z}_m))=\left\{
                  \begin{array}{ll}
                    \alpha_1+1, & \hbox{if $n= p_1p_2$, $m=p_1^{\alpha_1}p_2$ and $\alpha_1\geq 2$;} \\
                    \prod_{i=1}^{s}(\alpha_i-\beta_i+1)-1+s', & \hbox{otherwise.}
                  \end{array}
                \right.
                $
\end{center}}
\end{thm}

\begin{proof}
{ By Lemma \ref{n,d,isole}, it is clear that if $n= p_1p_2$, $m=p_1^{\alpha_1}p_2$ and $\alpha_1\geq 2$, then $\alpha(G_n(\mathbb{Z}_m))=\alpha_1+1$. Otherwise, with no loss of generality we may assume that  $S'=\{1,\ldots,s'\}$. Let $B=\mathfrak{A}\cup\left\{p_j^{\beta_j-1}\prod_{i\neq j} p_i^{\alpha_i}\mathbb{Z}_m : 1\leq j\leq s'\right\}$. Obviously, $n\nmid p_j^{\beta_j-1}\prod_{i\neq j} p_i^{\alpha_i}$, for every $j$, $1\leq j\leq s'$. Hence by Lemma \ref{n,d,isole}, $p_j^{\beta_j-1}\prod_{i\neq j} p_i^{\alpha_i}\mathbb{Z}_m$ is not an isolated vertex, for every $j$, $1\leq j\leq s'$. Also, it is easy to see that $B$ is an independent set and so $\alpha(G_n(\mathbb{Z}_m))\geq |\mathfrak{A}|+s'=\prod_{i=1}^{s}(\alpha_i-\beta_i+1)-1+s'$. Moreover, if $C$ is an independent set of $G_n(\mathbb{Z}_m)\setminus \mathfrak{A}$ and $|C|\geq s'+1$, then by Pigeonhole Principle we conclude that there exist $\prod_{i=1}^{s}p_i^{r_i}\mathbb{Z}_m,\prod_{i=1}^{s}p_i^{r'_i}\mathbb{Z}_m\in C$ such that $r_j,r'_j< \beta_j$, for some $j$, $1\leq j\leq s'$. This implies that $\prod_{i=1}^{s}p_i^{r_i}\mathbb{Z}_m$ and $\prod_{i=1}^{s}p_i^{r'_i}\mathbb{Z}_m$ are adjacent, which is impossible. Therefore
\begin{center}
$\alpha(G_n(\mathbb{Z}_m))= \prod_{i=1}^{s}(\alpha_i-\beta_i+1)-1+s'$.
\end{center}
}
\end{proof}

We denote $\{i\in S \ :\, r_i< \beta _i\}$ by $D_d$, where $d=p_1^{r_1}\cdots p_s^{r_s}$ is a divisor of $m$. Obviously, $D_d\subseteq S'$.
\begin{thm}\label{deg}
{ Let $\mathbb{Z}_n$ be a $\mathbb{Z}_m$-module and $d=p_1^{r_1}\cdots p_s^{r_s}(\neq 1,m)$ be a divisor of $m$. If $n$ divides $d$, then $deg(d\mathbb{Z}_m)=0$ and otherwise
 $$deg(d\mathbb{Z}_m)=\prod_{i=1}^{s}(\alpha_i+1)-2-\prod_{i\notin D_d}(\alpha_i+1)\prod_{i\in D_d}(\alpha_i-\beta_i+1).$$
 }
\end{thm}

\begin{proof}
{If $n|d$, then $d\mathbb{Z}_m$ is an isolated vertex and so $deg(d\mathbb{Z}_m)=0$. Assume that $n$ does not divide $d$. Then $D_d$ is nonempty. Clearly, $d\mathbb{Z}_m$ and $p_1^{t_1}\cdots p_s^{t_s}\mathbb{Z}_m$ are not adjacent if and only if $t_i\geq \beta_i$ for each $i\in D_d$. Thus the number of vertices not adjacent to $d\mathbb{Z}_m$ is $\prod_{i\notin D_d}(\alpha_i+1)\prod_{i\in D_d}(\alpha_i-\beta_i+1)-1$ and hence the number of its neighbors is $\prod_{i=1}^{s}(\alpha_i+1)-2-\prod_{i\notin D_d}(\alpha_i+1)\prod_{i\in D_d}(\alpha_i-\beta_i+1)$.
}
\end{proof}

\begin{thm}
\label{deltabozorg}
{ Suppose that $\mathbb{Z}_n$ is a $\mathbb{Z}_m$-module and $G_n(\mathbb{Z}_m)$ is not a null graph. If $n=p_1\cdots p_s$, then
$\Delta(G_n(\mathbb{Z}_m))=\prod_{i=1}^{s}(\alpha_i+1)-2-(\alpha_1+1)\prod_{i=2}^{s}\alpha_i$, where $\alpha_1\geq \cdots\geq \alpha_s$ and otherwise
 $\Delta(G_n(\mathbb{Z}_m))=\prod_{i=1}^{s}(\alpha_i+1)-2-\prod_{i=1}^{s}(\alpha_i-\beta_i+1)$.
}
\end{thm}

\begin{proof}
{ First, suppose that $n=p_1\cdots p_s$. By Theorem \ref{deg}, if $d=p_1^{r_1}\cdots p_s^{r_s}$ is a divisor of $m$ and $n$ does not divide $d$, then $deg(d\mathbb{Z}_m)=\prod_{i=1}^{s}(\alpha_i+1)-2-\prod_{i\notin D_d}(\alpha_i+1)\prod_{i\in D_d}\alpha_i$.
If $\alpha_1\geq \cdots\geq \alpha_s$, then $deg(d\mathbb{Z}_m)=\Delta(G_n(\mathbb{Z}_m))$ if and only if $\prod_{i\notin D_d}(\alpha_i+1)\prod_{i\in D_d}\alpha_i=(\alpha_1+1)\prod_{i=2}^{s}\alpha_i$. Let $\{i\in S : \alpha_i=\alpha_1 \}=\{1,\ldots,k\}$ for some $k$, $1\leq k\leq s$. In fact, $\left\{p_i^{r_i}\mathbb{Z}_m : 1\leq i\leq k, \ 1\leq r_i\leq \alpha_i\right\}$ is the set of all vertices with maximum degree. So $\Delta(G_n(\mathbb{Z}_m))=\prod_{i=1}^{s}(\alpha_i+1)-2-(\alpha_1+1)\prod_{i=2}^{s}\alpha_i$.\\
\indent Now, assume that there is an integer $j$ such that $\beta_j\neq 1$. We claim that there exist some vertices adjacent to all non-isolated vertices and so $\Delta(G_n(\mathbb{Z}_m))=\prod_{i=1}^{s}(\alpha_i+1)-2-\prod_{i=1}^{s}(\alpha_i-\beta_i+1)$. If $\beta_j=0$, then $p_j\mathbb{Z}_m$ is adjacent to all non-isolated vertices. Otherwise, $\beta_j\geq 2$. With no loss of generality suppose that $S'=\{1,\ldots,s'\}$. Let $d=p_1^{\beta_1-1}\cdots p_{s'}^{\beta_{s'}-1}$. Then $d\mathbb{Z}_m$ is adjacent to all non-isolated vertices. The claim is proved.
}
\end{proof}

\begin{thm}
\label{rasimotaselbahame}
{ Let $\mathbb{Z}_n$ be a $\mathbb{Z}_m$-module. Then $G_n(\mathbb{Z}_m)$ has a vertex which is adjacent to all other vertices if and only if $n=m$ and $\alpha_j\geq 2$, for some $j$, $1\leq j\leq s$.}
\end{thm}

\begin{proof}
{ If $G_n(\mathbb{Z}_m)$ has a vertex which is adjacent to all other vertices, then $G_n(\mathbb{Z}_m)$ has not any isolated vertex. This implies that $n=m$. By contradiction suppose that $\alpha_1=\cdots =\alpha_s=1$. Let $d\mathbb{Z}_m$ be a vertex of $G_n(\mathbb{Z}_m)$ such that it is adjacent to all other vertices. With no loss of generality, we may assume that $d=p_1\cdots p_t$, where $1\leq t<s$. Let $d'=p_{t+1}\cdots p_{s}$. It is easy to see that $d\mathbb{Z}_m$ and $d'\mathbb{Z}_m$ are non-adjacent, a contradiction. Therefore $\alpha_j\geq 2$, for some $j$, $1\leq j\leq s$. Conversely, if $n=m$ and $\alpha_j\geq 2$, for some $j$, $1\leq j\leq s$, then by Corollary \ref{noisolatedvertex}, $G_n(\mathbb{Z}_m)$ has no isolated vertex. Also, in view of the proof of Theorem \ref{deltabozorg}, we find that $G_n(\mathbb{Z}_m)$ has a vertex which is adjacent to all other vertices.}
\end{proof}

The following corollary is a generalization of \cite[Theorem 2.9]{Chakrabarty}.
\begin{cor}
{ Let $\mathbb{Z}_n$ be a $\mathbb{Z}_m$-module. Then $G_n(\mathbb{Z}_m)$ is a complete graph if and only if $n=m=p_1^{\alpha_1}$ and $\alpha_1\geq 2$.}
\end{cor}

\begin{proof}
{ Suppose that $G_n(\mathbb{Z}_m)$ is a complete graph. By Theorem \ref{rasimotaselbahame}, we find that $n=m$ and $\alpha_j\geq 2$, for some $j$, $1\leq j\leq s$. If $s\geq 2$, then $p_s^{\alpha_s}\mathbb{Z}_m$ and $\prod_{i=1}^{s-1} p_i^{\alpha_i}\mathbb{Z}_m$ are two non-adjacent vertices which is a contradiction. Therefore $n=m=p_1^{\alpha_1}$ and $\alpha_1\geq 2$. The other side is obvious.}
\end{proof}

\begin{thm}\label{gama}
{ Let $\mathbb{Z}_n$ be a $\mathbb{Z}_m$-module. If $G_n(\mathbb{Z}_m)$ is not a null graph, then $$\gamma(G_n(\mathbb{Z}_m))=\left\{
                  \begin{array}{ll}
                    |\mathfrak{A}|+1, & \hbox{if $n\neq p_1 \cdots p_s$ or $n= p_1p_2$, $m=p_1^{\alpha_1}p_2$ and $\alpha_1\geq 2$;} \\
                    |\mathfrak{A}|+2, & \hbox{otherwise.}
                  \end{array}
                \right.
                $$}
\end{thm}

\begin{proof}
{ Suppose that $n\neq p_1\cdots p_s$. In view of the proof of Theorem \ref{deltabozorg}, $G_n(\mathbb{Z}_m)$ has a vertex which is adjacent to every non-isolated vertices. This implies that $\gamma(G_n(\mathbb{Z}_m))=|\mathfrak{A}|+1$. Next, assume that $n= p_1\cdots p_s$. If $s=1$, then
$G_n(\mathbb{Z}_m)$ is a null graph. Now, assume that $s\geq 2$. If $n= p_1p_2$, $m=p_1^{\alpha_1}p_2$ and $\alpha_1\geq 2$, then by Lemma \ref{n,d,isole}, $\gamma(G_n(\mathbb{Z}_m))=|\mathfrak{A}|+1$. Otherwise, let $B=\{p_1\mathbb{Z}_m, p_2\cdots p_s\mathbb{Z}_m\}$. If $d\mathbb{Z}_m$ is a non-isolated vertex, then $n$ does not divide $d$. If $p_1\nmid d$, then $d\mathbb{Z}_m$ is adjacent to $p_2\cdots p_s\mathbb{Z}_m$. Otherwise, $p_j\nmid d$, for some $j$, $2\leq j\leq s$. Therefore $p_j\nmid [d,p_1]$. This yields that $d\mathbb{Z}_m$ is adjacent to $p_1\mathbb{Z}_m$. Therefore $B$ is a dominating set for $G_n(\mathbb{Z}_m)\setminus \mathfrak{A}$. Now, we claim that
$G_n(\mathbb{Z}_m)$ has not a vertex which is adjacent to every non-isolated vertices. By contradiction, suppose that $d\mathbb{Z}_m$ is adjacent to every non-isolated vertices. Since $n\nmid d$, we may assume that $d=p_1\cdots p_t$, where $1\leq t<s$. Let $d'=p_{t+1}\cdots p_s$. Clearly, $d\mathbb{Z}_m$ and $d'\mathbb{Z}_m$ are non-adjacent. Since $n\nmid d$, by Theorem \ref{n,d,isole}, $d'\mathbb{Z}_m$ is not an isolated vertex, a contradiction. Thus $\gamma(G_n(\mathbb{Z}_m))=|\mathfrak{A}|+2$.}
\end{proof}

\section{Chromatic Index of $G_n(\mathbb{Z}_m)$}

In this section, we study the chromatic index of $G_n(\mathbb{Z}_m)$. First, we need the following theorems:

\begin{pro}
\label{vizing}
{\rm \cite[Theorem 17.4]{bondy} (Vizing's Theorem) If $G$ is a simple graph, then either $\chi'(G)=\Delta(G)$ or $\chi'(G)=\Delta(G)+1$.}
\end{pro}

\begin{pro}
\label{maximumedgechromatic}
{\rm \cite[Corollary 5.4]{Beineke} Let $G$ be a simple graph. Suppose that for every vertex $u$ of maximum degree, there exists an edge $u$ --- $v$ such that
$\Delta(G)-deg(v)+2$ is more than the number of vertices with maximum degree in $G$. Then $\chi'(G)=\Delta(G)$.}
\end{pro}

\begin{pro}
\label{2s+1}
{\rm \cite[Theorem D]{Plantholt} If $G$ has order $2k$ and maximum degree $2k-1$, then $\chi'(G)=\Delta(G)$. If $G$ has order $2k + 1$ and maximum degree $2k$, then
$\chi'(G)=\Delta(G)+1$ if and only if the size of $G$ is at least $2k^{2}+1$.}
\end{pro}

\begin{thm}
{Suppose that $\mathbb{Z}_n$ is a $\mathbb{Z}_m$-module and $G_n(\mathbb{Z}_m)$ is not a null graph. If $n=p_1^{\beta_1}$, then $\chi'(G_n(\mathbb{Z}_m))=\Delta(G_n(\mathbb{Z}_m))$ if and only if $\beta_1\prod_{i=2}^{s}(\alpha_i+1)$ is an odd integer. }
\end{thm}

\begin{proof}
{ Clearly, the set of all non-isolated vertices of $G_n(\mathbb{Z}_m)$ is equal to $\{d\mathbb{Z}_m\,:\, d=p_1^{r_1}\cdots p_s^{r_s}, d\neq1, \ 0\leq r_1\leq \beta_1-1 \ \hbox{and}\ 0\leq r_i\leq \alpha_i \ \hbox{for}\ i=2,\ldots,s \}$. So $G_n(\mathbb{Z}_m)\setminus \mathfrak{A}$ is a complete graph of order $\beta_1\prod_{i=2}^{s}(\alpha_i+1)-1$. The result, now, follows from \cite[Theorem 5.11]{and}.}
\end{proof}

\begin{thm}\label{alfa1,2}
{ Suppose that $\mathbb{Z}_n$ is a $\mathbb{Z}_m$-module and $G_n(\mathbb{Z}_m)$ is not a null graph. If $n=p_1p_2$ and $m=p_1^{\alpha_1}p_2^{\alpha_2}$, then $\chi'(G_n(\mathbb{Z}_m))=\Delta(G_n(\mathbb{Z}_m))$ if and only if $max\{\alpha_1, \alpha_2\}$ is an even integer. }
\end{thm}

\begin{proof}
{ If $\alpha_1=\alpha_2=1$, then $G_n(\mathbb{Z}_m)$ is a null graph. If $\alpha_1=1$ and $\alpha_2\geq 2$, then $G_n(\mathbb{Z}_m)\cong K_{\alpha_2}\cup \overline{K_{\alpha_2}}$.
If $\alpha_2,\alpha_2\geq 2$, then assume that $B_i=\{p_i^{r_i}\mathbb{Z}_m\,:\,1\leq r_i\leq \alpha_i \}$, for $i=1,2$. Clearly, $B_1\cup B_2$ is the set of all non-isolated vertices of $G_n(\mathbb{Z}_m)$, every $B_i$ forms a complete graph and $G_n(\mathbb{Z}_m)\setminus \mathfrak{A}$ is the disjoint union of them. Therefore $G_n(\mathbb{Z}_m)\setminus \mathfrak{A}\cong K_{\alpha_1}\cup K_{\alpha_2}$. Now, \cite[Theorem 5.11]{and} completes the proof.}
\end{proof}

\begin{thm}
{ Let $\mathbb{Z}_n$ be a $\mathbb{Z}_m$-module. If $n=p_1\cdots p_s$ and $s\geq 3$, then $\chi'(G_n(\mathbb{Z}_m))=\Delta(G_n(\mathbb{Z}_m)).$}
\end{thm}

\begin{proof}
{ With no loss of generality assume that $\alpha_1\geq \cdots\geq\alpha_s$. Let $\{i\in S : \alpha_i=\alpha_1 \}=\{1,\ldots,k\}$ for some $k$, $1\leq k\leq s$. By Theorem \ref{deltabozorg}, we know that $\{p_i^{r_i}\mathbb{Z}_m : 1\leq i\leq k, \ 1\leq r_i\leq \alpha_i\}$ is the set of all vertices with maximum degree. Hence $G_n(\mathbb{Z}_m)$ has $k\alpha_1$ vertices with maximum degree. On the other hand, we know that $\Delta(G_n(\mathbb{Z}_m))=\prod_{i=1}^{s}(\alpha_i+1)-2-(\alpha_1+1)\prod_{i=2}^{s}\alpha_i$ and $deg(p_1\cdots p_{s-1}\mathbb{Z}_m)= \prod_{i=1}^{s}(\alpha_i+1)-2-\alpha_s\prod_{i=1}^{s-1}(\alpha_i+1)$. It is easy to check that if $s\geq 4$, then $\Delta(G_n(\mathbb{Z}_m))-deg(p_1\cdots p_{s-1}\mathbb{Z}_m) +2$ is more than $\alpha_1+\cdots +\alpha_s$ and so $\Delta(G_n(\mathbb{Z}_m))-deg(p_1\cdots p_{s-1}\mathbb{Z}_m) +2$ is more than the number of vertices with maximum degree. Note that $deg(p_s^{r_s}\mathbb{Z}_m)=\Delta(G_n(\mathbb{Z}_m))$ ($1\leq r_s\leq \alpha_s$) if and only if $\alpha_1=\cdots=\alpha_s$. Also, if $s\geq 4$, then $\Delta(G_n(\mathbb{Z}_m))-deg(p_2\cdots p_{s}\mathbb{Z}_m) +2$ is more than $s\alpha_1$. Therefore by Theorem \ref{maximumedgechromatic}, we find that $\chi'(G_n(\mathbb{Z}_m))=\Delta(G_n(\mathbb{Z}_m))$. Now, consider $s=3$. If $\alpha_3\geq 2$, then it is easy to see that $\Delta(G_n(\mathbb{Z}_m))-deg(p_1\cdots p_{s-1}\mathbb{Z}_m) +2$ is more than $\alpha_1+\alpha_2+\alpha_3$ and the proof is complete. Therefore assume that $\alpha_3=1$. There are two following cases:\\
\indent{Case 1.} $\alpha_2=1$. If $\alpha_1=1$, then $n=m$ and so $G_n(\mathbb{Z}_m)=G(\mathbb{Z}_m)$. One can easily check that $\chi'(G(\mathbb{Z}_{p_1p_2p_3}))=\Delta(G(\mathbb{Z}_{p_1p_2p_3}))=4$ (see \cite[Fig. 3]{Chakrabarty}). If $\alpha_1\geq 2$, then $G_n(\mathbb{Z}_m)$ has $\alpha_1$ vertices with maximum degree. Also, $\Delta(G_n(\mathbb{Z}_m))-deg(p_1 p_{3}\mathbb{Z}_m) +2$ is more than $\alpha_1$ and the result, follows from Theorem \ref{maximumedgechromatic}.\\
\indent{Case 2.} $\alpha_2\geq 2$. In this case, it is easy to see that $\Delta(G_n(\mathbb{Z}_m))-deg(p_2 p_{3}\mathbb{Z}_m) +2$ is more than $\alpha_1+\alpha_2$. Also, $\Delta(G_n(\mathbb{Z}_m))-deg(p_1 p_{3}\mathbb{Z}_m) +2$ is more than $\alpha_1+\alpha_2$ and the result, follows from Theorem \ref{maximumedgechromatic}.}
\end{proof}

\begin{thm}
{ Let $\mathbb{Z}_n$ be a $\mathbb{Z}_m$-module. If $n\neq p_1\cdots p_s$ and $s\geq 2$, then $\chi'(G_n(\mathbb{Z}_m))=\Delta(G_n(\mathbb{Z}_m)).$}
\end{thm}

\begin{proof}
{ By Theorem \ref{deltabozorg}, we know that $G_n(\mathbb{Z}_m)$ has $\prod_{i=1}^{s}(\alpha_i+1)-1-\prod_{i=1}^{s}(\alpha_i-\beta_i+1)$ non-isolated vertices and $\Delta(G_n(\mathbb{Z}_m))=\prod_{i=1}^{s}(\alpha_i+1)-2-\prod_{i=1}^{s}(\alpha_i-\beta_i+1)$. Hence by Theorem \ref{2s+1}, we find that $\chi'(G_n(\mathbb{Z}_m))=\Delta(G_n(\mathbb{Z}_m))$, where $\prod_{i=1}^{s}(\alpha_i+1)-1-\prod_{i=1}^{s}(\alpha_i-\beta_i+1)$ is even. Next, assume that $\prod_{i=1}^{s}(\alpha_i+1)-1-\prod_{i=1}^{s}(\alpha_i-\beta_i+1)$ is odd. Since the size of
a complete graph of order $2k+1$ is $2k^{2}+k$, if we prove that $G_n(\mathbb{Z}_m)$, in this case, losses at least $k=\big(\prod_{i=1}^{s}(\alpha_i+1)-2-\prod_{i=1}^{s}(\alpha_i-\beta_i+1)\big)/2$ edges, then by Theorem \ref{2s+1},
$\chi'(G_n(\mathbb{Z}_m))=\Delta(G_n(\mathbb{Z}_m))$. Let $d=\prod_{i=1}^{s}p_i^{r_i}$ be a divisor of $m$ such that $n\nmid d$ and $\{i\in S' : r_i\geq \beta _i\}\neq \varnothing$. Let $\{i\in S' : r_i\geq \beta _i\}=\{1,\ldots,t\}$. Set $\overline{d}=\prod_{i=t+1}^{s'}p_i^{\beta_i}\prod_{i=s'+1}^{s}p_i^{r_i}$. It is easy to check that $d\mathbb{Z}_m$ and $\overline{d}\mathbb{Z}_m$ are not adjacent. So we conclude that $G_n(\mathbb{Z}_m)$ losses at least $\big(\prod_{i=1}^{s}(\alpha_i+1)-\prod_{i=1}^{s}(\alpha_i-\beta_i+1)-\prod_{i=1}^{s'}\beta_i\prod_{i=s'+1}^{s}(\alpha_i+1)\big)/2$ edges. We continue
the proof in the following two cases:\\
\indent{Case 1.} $\beta_i\geq 2$, for some $i$, $1\leq i\leq s'$. With no loss of generality we may assume that $\beta_{s'}\geq 2$. Suppose that $B=\left\{p_1^{r_1}\prod_{i=2}^{s'}p_i^{\beta_i}\mathbb{Z}_m : 0\leq r_1\leq \beta_1-1\right\}$ and
$C=\left\{p_1^{\beta_1}\prod_{i=2}^{s}p_i^{r_i}\mathbb{Z}_m : 0\leq r_i\leq \beta_i-1, \ \hbox{for} \ 2\leq i\leq s'  \ \hbox{and otherwise} \ 0\leq r_i\leq \alpha_i\right\}.$ Clearly, every element of $B$ is not adjacent to every element of $C$. This implies that $G_n(\mathbb{Z}_m)$ losses at least $\prod_{i=1}^{s'}\beta_i\prod_{i=s'+1}^{s}(\alpha_i+1)- (\prod_{i=2}^{s'}\beta_i+ \beta_1-1)$ new edges. Therefore $G_n(\mathbb{Z}_m)$ losses at least $l=(\prod_{i=1}^{s}(\alpha_i+1)-\prod_{i=1}^{s}(\alpha_i-\beta_i+1)+\prod_{i=1}^{s'}\beta_i\prod_{i=s'+1}^{s}(\alpha_i+1))/2- (\prod_{i=2}^{s'}\beta_i+ \beta_1-1)$ edges. It suffices we prove that $k\leq l$. One can easily see that $k\leq l$ if and only if $2(\beta_1-2)\leq (\beta_1\prod_{i=s'+1}^{s}(\alpha_i+1)-2)\prod_{i=2}^{s'}\beta_i$. Since $\beta_{s'}\geq 2$, so $2(\beta_1-2)\leq (\beta_1\prod_{i=s'+1}^{s}(\alpha_i+1)-2)\prod_{i=2}^{s'}\beta_i$ and hence $k\leq l$.\\
\indent{Case 2.} $\beta_1=\cdots=\beta_{s'}=1$. Clearly, $\prod_{i=2}^{s'}p_i\mathbb{Z}_m$ and  $p_1\prod_{i=s'+1}^{s}p_i^{r_i}\mathbb{Z}_m$ are non-adjacent, where $0\leq r_i\leq \alpha_i,$ for $i=s'+1,\ldots,s$ . This implies that $G_n(\mathbb{Z}_m)$ losses at least $\prod_{i=s'+1}^{s}(\alpha_i+1)- 1$ new edges. Therefore $G_n(\mathbb{Z}_m)$ losses at least $l=(\prod_{i=1}^{s}(\alpha_i+1)-\prod_{i=1}^{s}(\alpha_i-\beta_i+1)+\prod_{i=s'+1}^{s}(\alpha_i+1))/2- 1$ edges. In this case, $k\leq l$ is obvious and the proof is complete.}
\end{proof}

From the above theorems, we can deduce the next result.
\begin{cor}{ Let $\mathbb{Z}_n$ be a $\mathbb{Z}_m$-module. If $G_n(\mathbb{Z}_m)$ is not a null graph, then $\chi'(G_n(\mathbb{Z}_m))=\Delta(G_n(\mathbb{Z}_m))$, unless the following cases:
\par $(i)$ $n=p_1^{\beta_1}$ and $\beta_1\prod_{i=2}^{s}(\alpha_i+1)$ is even.
\par $(ii)$ $n=p_1p_2$, $m=p_1^{\alpha_1}p_2^{\alpha_2}$ and $max\{ \alpha_1,\alpha_2\}$ is odd.}
\end{cor}

\section{ Eulerian Tour in $G_n(\mathbb{Z}_m)$}
An \textit{Eulerian tour} in a graph is a closed trail including all the edges of the
graph. A graph is \textit{Eulerian} if it has an Eulerian tour. By \cite[Theorem 4.1]{bondy}, a simple connected graph is Eulerian if and only if it has
no vertices of odd degree. In this section, we determine all integer numbers $n$ and $m$ for which $G_n(\mathbb{Z}_m)\setminus \mathfrak{A}$ is an Eulerian graph. We start with the following theorem.

\begin{thm}
{ Let $\mathbb{Z}_n$ be a $\mathbb{Z}_m$-module. If $G_n(\mathbb{Z}_m)$ is not a null graph, then $diam(G_n(\mathbb{Z}_m)\setminus \mathfrak{A})\leq 4$, except for the case $n=p_1p_2$, $m=p_1^{\alpha_1}p_2^{\alpha_2}$ and $\alpha_1,\alpha_2\geq2$, in which case $G_n(\mathbb{Z}_m)\setminus \mathfrak{A}\cong K_{\alpha_1}\cup K_{\alpha_2}$.}
\end{thm}

\begin{proof}
{ Assume that $G_n(\mathbb{Z}_m)$ is not a null graph. In view of the proof of Theorem \ref{gama}, we conclude that if $n\neq p_1 \cdots p_s$, then $G_n(\mathbb{Z}_m)\setminus \mathfrak{A}$ is connected and $diam(G_n(\mathbb{Z}_m)\setminus \mathfrak{A})\leq 2$. Now, suppose that $n=p_1 \cdots p_s$. Since $G_n(\mathbb{Z}_m)$ is not a null graph, so $s\neq 1$. First assume that $s\geq 3$. By the proof of Theorem \ref{gama}, we know that $\{p_1\mathbb{Z}_m, p_2\cdots p_s\mathbb{Z}_m\}$ is a dominating set for $G_n(\mathbb{Z}_m)\setminus \mathfrak{A}$. Since $s\geq 3$, $p_2\mathbb{Z}_m$ is adjacent to both $p_1\mathbb{Z}_m$ and $ p_2\cdots p_s\mathbb{Z}_m$. This implies that $G_n(\mathbb{Z}_m)\setminus \mathfrak{A}$ is connected and $diam(G_n(\mathbb{Z}_m)\setminus \mathfrak{A})\leq 4$. Next, assume that $s=2$. By Lemma \ref{n,d,isole}, if $n= p_1p_2$, $m=p_1^{\alpha_1}p_2$ and $\alpha_1\geq 2$, then $diam(G_n(\mathbb{Z}_m)\setminus \mathfrak{A})=1$. Now, consider $n=p_1p_2$, $m=p_1^{\alpha_1}p_2^{\alpha_2}$ and $\alpha_1,\alpha_2\geq2$. As we saw in the proof of Theorem \ref{alfa1,2}, $G_n(\mathbb{Z}_m)\setminus \mathfrak{A}\cong K_{\alpha_1}\cup K_{\alpha_2}$ and the proof is complete.}
\end{proof}

Now, we are in a position to generalize Theorem 5.1 of \cite{Chakrabarty}.
\begin{thm}
{ Suppose that $\mathbb{Z}_n$ is a $\mathbb{Z}_m$-module and $G_n(\mathbb{Z}_m)$ is not a null graph. Then $G_n(\mathbb{Z}_m)\setminus \mathfrak{A}$ is an Eulerian graph if and only if one of the following holds:
\par $(i)$ $\alpha_i$ and $\beta_i$ are even integers for each $i$, $1\leq i\leq s$.
\par $(ii)$ $\alpha_i$ is an odd integer and $\beta_i$ is an even integer for some $i$, $1\leq i\leq s$.
\par $(iii)$ $n=p_1\cdots p_s$ and $m=p_1^{\alpha_1}\cdots p_s^{\alpha_s}$, where $\alpha_i$' are odd integers and $s\geq 3$.
\par $(iv)$ $n=p_1p_2$ and $m=p_1^{\alpha_1}p_2$, where $\alpha_1>1$ is an odd integer.
}
\end{thm}

\begin{proof}
{ By Theorem \ref{deg}, $G_n(\mathbb{Z}_m)\setminus \mathfrak{A}$ is an Eulerian graph if and only if for each non-isolated vertex $d\mathbb{Z}_m$ of $G_n(\mathbb{Z}_m)$ both $\prod_{i=1}^{s}(\alpha_i+1)$ and $\prod_{i\notin D_d}(\alpha_i+1)\prod_{i\in D_d}(\alpha_i-\beta_i+1)$ are even or odd integers. So the ``if" part of theorem is obvious. For the converse, suppose that $G_n(\mathbb{Z}_m)\setminus \mathfrak{A}$ is an Eulerian graph. First assume that both $\prod_{i=1}^{s}(\alpha_i+1)$ and $\prod_{i\notin D_d}(\alpha_i+1)\prod_{i\in D_d}(\alpha_i-\beta_i+1)$ are odd integers for a vertex $d\mathbb{Z}_m$ of $G_n(\mathbb{Z}_m)\setminus \mathfrak{A}$. Thus all $\alpha_i$' are even integers. If $\beta_i$ is an odd integer for some $i$, $1\leq i\leq s$, then there exists a vertex $d'\mathbb{Z}_m$ such that $D_{d'}=\{i\}$. (Note that $G_n(\mathbb{Z}_m)$ is not a null graph.) So $\prod_{i\notin D_{d'}}(\alpha_i+1)\prod_{i\in D_{d'}}(\alpha_i-\beta_i+1)$ is an even integer which implies that $deg(d'\mathbb{Z}_m)$ is an odd integer, a contradiction. Therefore $\beta_i$ is an even integer for each $i$, $1\leq i\leq s$.
Next, assume that both $\prod_{i=1}^{s}(\alpha_i+1)$ and $\prod_{i\notin D_d}(\alpha_i+1)\prod_{i\in D_d}(\alpha_i-\beta_i+1)$ are odd integers for a vertex $d\mathbb{Z}_m$ of $G_n(\mathbb{Z}_m)\setminus \mathfrak{A}$. Hence $\alpha_i$ is an odd integer for some $i$, $1\leq i\leq s$. With no loss of generality suppose that $\{i\in S : \alpha_i \ \hbox{is an odd integer}\}=\{1,\ldots,t\}$, where $1\leq t\leq s$. Suppose that $\beta_1,\ldots,\beta_t$ are odd integers. If $n=p_1\cdots p_s$ and $m=p_1^{\alpha_1}\cdots p_s^{\alpha_s}$, where $\alpha_i$' are odd integers, then we are done. (Note that $G_n(\mathbb{Z}_m)\setminus \mathfrak{A}$ is a connected graph.) Otherwise there exists a vertex $d'\mathbb{Z}_m$ such that $D_{d'}=\{1,\ldots,t\}$. So $\prod_{i\notin D_{d'}}(\alpha_i+1)\prod_{i\in D_{d'}}(\alpha_i-\beta_i+1)$ is an odd integer which implies that $deg(d'\mathbb{Z}_m)$ is an odd integer, a contradiction. Thus $\beta_i$ is an even integer for some $i$, $1\leq i\leq t$. The proof is complete.}
\end{proof}

\end{document}